\documentclass[imslayout]{imsart}

\RequirePackage[OT1]{fontenc}
\RequirePackage{amsthm,amsmath,amssymb}
\RequirePackage[numbers]{natbib}

\usepackage{tikz}
\usetikzlibrary{positioning}
\usetikzlibrary{calc}
\usetikzlibrary{intersections}
\usetikzlibrary{decorations} 
\usetikzlibrary{decorations.pathmorphing} 
\usepackage{tikz-3dplot}
\usepackage{pgfplots}

\startlocaldefs

\newtheorem{theorem}{\bf Theorem}[section]
\newtheorem{lemma}[theorem]{\bf Lemma}
\newtheorem{corollary}[theorem]{\bf Corollary}
\newtheorem{definition}[theorem]{\bf Definition}
\newtheorem{remark}[theorem]{\bf Remark}
\newtheorem{proposition}[theorem]{ Proposition}
\newtheorem{example}[theorem]{\bf Example}
\newcommand{\Real}{\mathbb{R}}

\newcommand{\ones}{\mathbf{1}}
\newcommand{\eps}{\varepsilon}

\newcommand{\pmx}[1]{
  \begin{pmatrix}
    #1
  \end{pmatrix}
}

\DeclareMathOperator{\E}{E}
\DeclareMathOperator{\Prob}{P}

\DeclareMathOperator{\card}{card}
\DeclareMathOperator{\diag}{diag}
\DeclareMathOperator{\sign}{sign}
\DeclareMathOperator{\opspan}{span}
\DeclareMathOperator{\opnull}{null}
\DeclareMathOperator{\cl}{cl}
\DeclareMathOperator{\conv}{conv}
\DeclareMathOperator{\tr}{tr}
\DeclareMathOperator{\DoF}{dof}
\DeclareMathOperator{\gdof}{GDoF}

\endlocaldefs

\begin{document}

\begin{frontmatter}

\title{The geometry of PLS shrinkages}
\runtitle{PLS geometry}

\begin{aug}
  \author{\fnms{Paolo} \snm{Foschi}%
    \ead[label=e1]{paolo.foschi2@unibo.it}},

  \runauthor{P. Foschi}

  \affiliation{University of Bologna}

  \address{Dept. of Statistical Sciences, Via Belle Arti, 41, 40126
    Bologna, Italy,\\ \printead{e1}}

\end{aug}


\begin{abstract}
  The geometrical structure of PLS shrinkages is here considered.
  Firstly, an explicit formula for the shrinkage vector is provided.
  In that expression, shrinkage factors are expressed a averages of a
  set of basic shrinkages that depend only on the data matrix. On the
  other hand, the weights of that average are multilinear functions of
  the observed responses. That representation allows to characterise
  the set of possible shrinkages and identify extreme situations where
  the PLS estimator has an highly nonlinear behaviour. In these
  situations, recently proposed measures for the degrees of freedom
  (DoF), that directly depend on the shrinkages, fail to provide
  reasonable values. It is also shown that the longstanding conjecture
  that the DoFs of PLS always exceeds the number PLS directions does
  not hold. 
\end{abstract}



\end{frontmatter}

\section{Introduction}

In the last decades, the Partial Least Squares (PLS) methodology has
gained high popularity among data analysts and statisticians.
That approach has been applied to several statistical and data
analysis problems, ranging from univariate regressions to more complex
models involving multivariate responses, latent variables or
functional data, to cite a few. Nonetheless, even for its simplest
application, namely PLS regressions, only few fundamental results have
been obtained. For instance, apart for an asymptotic approximation
derived by means of a delta method (see \cite{Denham:97,Phatak:2002}),
the distribution or the moments of PLS estimator are still
unknown. The difficulty on tackling this task is testified by small
amount of work that has been published on the top journals of
methodological statistics
\cite{ButlerDenham:2000,ChunKeles:2010,CookHellandSu:2013,%
  DelaigleHall:2012,FrankFriedman:1993,Goutis:96,Goutis:96b,%
  KramerSugiyama:2011,NaikTsai:2000,ReissOgden:2007}. It worth adding
that, a part of this set of these papers deals with extensions of the PLS
approach, for instance to infinite dimensional spaces, instead of
investigating the mathematical and inferential structure
\cite{DelaigleHall:2012,Goutis:96b,ReissOgden:2007}.

In parallel to statisticians, the numerical linear algebra community
have studied the same kind of tools calling them Krylov or Conjugate
Gradient methods (see \cite{LiesenStrakos:book}). Their main concern
was the design of computationally efficient algorithms and the study
of numerical stability properties, which often are not too good for
these methods. Only about a decade ago, these two research streams
have been bridged (see \cite{Elden:2004,PhatakDeHoog:2002}). The work
here proposed takes inspiration from a paper which, firstly in that
community, recognised that the data matrix does not completely defines
the properties of the estimator, the final word is left to the
observation vector \cite{Greenbaum:96}.
Interestingly, despite the deep technical knowledge of these family of
methods, to the knowledge of the author, the numerical linear algebra
community have not yet recognised the key role played by the so called
shrinkage factors
\cite{ButlerDenham:2000,FrankFriedman:1993,Goutis:96,%
  Helland:1988,LingjaerdeChristophersen:2000,Kramer:2007}.

The analysis proposed in this paper grounds on the results available
on the PLS shrinkages and tries to make a step forward in the
understanding of the inner structure of the PLS regression
estimators. Firstly, a novel expression for the vector of shrinkages
which is explicit in terms of the observations is derived.  A similar
expression was proposed in \cite{ButlerDenham:2000}, but only for a
couple of special cases, and an analogous formula, but involving Ritz
values, was derived in \cite{LingjaerdeChristophersen:2000}. The
latter, however, is not fully explicit, being these Ritz values only
implicitly defined in terms of the observation vector.

By means of that expression the geometry of PLS shrinkage is formally
characterised. The range of all possible values for the shrinkage
vector is provided. This analysis encompass the one presented in
\cite{ButlerDenham:2000} where shrinkage and expansion patterns were
considered. It allows also to complete the work of Lingj{\ae}rde and
Christophersen in \cite{LingjaerdeChristophersen:2000}, where extreme
expansion or shrinkage behaviours are studied. In their concluding
table, Lingj{\ae}rde and Christophersen were not able to establish
bounds for one of the four extreme cases considered.
Sometimes, it is also presumed that shrinkage factors cannot be
negative For instance, this eventuality was not taken into account in
the analysis of Butler and Denham (see \cite{ButlerDenham:2000} page
588), even though it was previously considered in
\cite{FrankFriedman:1993} (see also \cite{Kramer:2007}). The same
oversight was made in the conclusions of
\cite{LingjaerdeChristophersen:2000}, where the authors presumed that
it was sufficient to bound shrinkages to one to avoid a ``harmful''
expansion. Here, it is shown, by an example, that large expansions
along principal directions arise even when the observation vector is
not orthogonal to those directions.

The possibility to have very large expansions, which is essentially
due the highly nonlinear nature of the estimator, has serious
consequences on the so-called Generalised Degrees of Freedom (GDoF), a
statistic proposed as an extension of the DoF concept to nonlinear
estimators \cite{Ye:93,Efron:2004}. That statistic was later applied
to PLS regressions in \cite{KramerSugiyama:2011}, where a conjecture,
originally formulated in \cite{FrankFriedman:1993,MartensNaes:book}
that states that the DoF of PLS are always larger than the number of
PLS directions, is supported. That statement was formally proven for
the single direction case and experimentally verified for the general
case. According to Kramer and Sugiyama, their tests ``confirmed'' that
conjecture \cite{KramerSugiyama:2011}.  It is worth noting that that
their experiments failed to support that conjecture for large models,
but the authors ascribed these negative perfomances to numerical
rounding errors.  Alternative measures for the DoF of PLS parameter
estimators have also been proposed in
\cite{Denham:97,Phatak:2002,VanDerVoet:1999}. These statistics,
however, seems to have even worse performances that the GDoF one.

Before concluding this short literature review, it should also
mentioned that an interesting alternative approach for the analysis of
PLS is proposed in \cite{Druilhet2008}, where the shrinkage
properties are studied along directions that differ from the
principal ones.

\medskip

This paper is structured as follows. Firstly, in the next section, the
PLS regression estimator is formulated as restricted least squares
estimator on a Krylov supspace. After a rotation on the principal
axes, the Krylov matrix is factorised as a diagonal by Vandermonde
matrix product.  This decomposition allows to separate the effects of
the response vector from those of the singular values of the data
matrix. In that section the main results are presented.  In
particular, a novel explicit expression where the shrinkage vector is
characterised as an average of a set of ``extreme'' shrinkages that do
not depend on the observations is provided. This expression allows to
geometrically characterise the set of possible shrinkages in terms of
these extreme points.  The third section contains formal proofs of
these results and a precise description and derivation of that
geometrical structure.  Finally the fourth section contains some
examples and a discussion on the obtained results. These examples will
show that an odd behaviour can be expected from PLS even in non
extreme setups. Furthermore, a sufficient condition and
counterexamples that invalidates the above mentioned conjecture on the
DoF of PLS regression are presented.

\section{Shrinkages for PLS regressions}
\label{sec:main}

Consider the estimation by means of Partial Least Squares regressions
of the following linear model
\begin{align*}
  \tilde{y} &= X \tilde\beta + \varepsilon, 
  &
  \varepsilon &\sim (0, \sigma^2 I).
\end{align*}
where $X \in \Real^{N \times m}$ is the regressor matrix,
$\tilde{y} \in \Real^{N}$ is the response vector 
and $\varepsilon$ the disturbance vector \cite{Goutis:96}.
Without loss of generality, the study of the partial least squares
(PLS) regression can be performed by considering its projection on the
principal axis of the regression matrix 
\cite{ButlerDenham:2000,Goutis:96,LingjaerdeChristophersen:2000}.
After a rotation, the normal equations associated to the above
regression model can be rewritten as
\begin{align}\label{eq:model}
  y &= \Lambda \beta + u,
  &
  u &\sim (0,\sigma^2 \Lambda),
\end{align}
where $\Lambda$ comes from the singular value decomposition
$X=U\Lambda^{\frac12}V^T$, $y = V^T X^T \tilde{y}$,
$\beta = V^T \tilde\beta$ and $u = V^T X^T \varepsilon$.  Here,
without loss of generality, the eigenvalues $\lambda_i$
($i=1,\ldots,m$) are assumed to be strictly positive, distinct and in
decreasing order: $\lambda_1 > \lambda_2 > \cdots > \lambda_m >0$.
Hereafter, a vector of all ones is denoted by $\ones$ and variables
represented by capital letters denote diagonal matrices generated from
vectors indicated by the corresponding lower-case variables, that is
$Y = \diag(y)$, $Z = \diag(z)$, $\Lambda = \diag(\lambda)$,
$\Psi = \diag(\psi)$, $\Omega = \diag(\omega)$ and so on. 

The PLS estimator with $n < m$ directions is given by
\begin{align*}
  \hat{\beta} &= K (K^T \Lambda K)^{-1} K^T y,
\end{align*}
where $K$ is the $m \times n$ Krylov matrix
\begin{align*}
  K &= \pmx{ y & \Lambda y & \cdots & \Lambda^{n-1} y},
\end{align*}
and it is assumed that $K^T \Lambda K$ is non-singular
\cite{ButlerDenham:2000,Helland:1988,Helland:1990,LingjaerdeChristophersen:2000}.
The PLS prediction and the associated residuals for $y$ are given by
$\hat{y} = P y$ and $r = (I-P)y$ with $P$ denoting the oblique
projection $P = \Lambda K (K^T \Lambda K)^{-1}K^T$.

It is convenient to factorise the Krylov matrix $K$ as
\begin{align}\label{eq:BK}
  K &= Y V,
\end{align}
where $V$ is the $m \times n$ Vandermonde matrix associated to
$\lambda$ given by
$V = \big( \ones \;\; \Lambda \ones \;\; \cdots \;\; \Lambda^{n-1} \ones \big)$.
Then, the projection matrix $P$ can be rewritten as
\begin{align*}
  P &= Y \Lambda V (V^T Y^2 \Lambda V)^{-1} V^T Y.
\end{align*}
Given the above assumptions on $\lambda$, that expression is well
posed, that is $K^T\Lambda K = V^T Y^2 \Lambda V$ is non-singular,
whenever $y$ has $n$ or more non-zero elements.

Shrinkage factors have been used in
\cite{ButlerDenham:2000,LingjaerdeChristophersen:2000} to study the
characteristics of PLS regression parameter estimates. Shrinkages are
defined as the ratios of the PLS estimated coefficients over the OLS
coefficients along principal axes. Since the OLS estimator of the
$i$-th coefficient is given by $y_i/\lambda_i$, the $i$-th shrinkage
is given by $\omega_i = \lambda_i \hat\beta_i/y_i$ and the vector of
shrinkages can be written as
\begin{align}\label{eq:shrinkDef}
  \omega &= Q\ones,  
           &
  Q &= \Lambda V (V^T \Psi \Lambda V)^{-1} V^T \Psi,
\end{align}
where $\Psi = Y^2$.  Here, $Q$ is the oblique projection on the range
of $\Lambda V$ along the null-space of $\Psi V$. Note that, defining
the shrinkages directly by \eqref{eq:shrinkDef} is more robust as it
allows for zero elements in the response vector $y$. Again, for $Q$ to
be well defined, $y$ needs to have at least $n$ non-zero elements.

A couple of properties can be immediately drawn from
\eqref{eq:shrinkDef}. Firstly, the shrinkage vector $\omega$ is
invariant to rescaling of the observation vector $y$ and secondly, it
does not depend on the signs the elements of $y$. Moreover, the
shrinkage vector $\omega$ belongs to the $n$-dimensional linear
manifold spanned by the columns of $\Lambda V$.

\subsection{Main results}

The task of obtaining simple and explicit expressions for the elements
of the projection matrices $Q$ and $P$ and of the shrinkages $\omega$
is rather difficult. However, these expressions can be obtained in
some special cases, for instance when the cardinality of $y$ is
exactly $n$, the number of PLS directions.  The following results
characterise the shrinkages in that case and in the general case.

Firstly, it is convenient to introduce some additional notation. The
set of all subsets of $S$ with cardinality $n$ is denoted by
$\binom{S}{n}$, the set of the first $m$ integers is denoted by
$[m] = \{1,2,\ldots,m\}$ and $[m,n]$ denotes the set of $n$-subsets of
$[m]$, that is $[m,n]= \binom{[m]}{n}$. For a set of indices
$\tau \in [m,n]$ and $x \in \Real^m$, $x^\tau$ and $x_\tau$ denote,
respectively, the monomial $x^\tau = \prod_{i\in\tau} x_i$ and the
subvector of $x$ obtained by selecting the elements in the positions
indicated by $\tau$. The $m \times n$ selection matrix associated to
that subsetting operation will be denoted by $S_\tau$:
$x_\tau = S_\tau^T x$. Moreover, the sets of non-negative and of
positive reals will be denoted by $\Real_+$ and $\Real_{++}$.

\begin{lemma}\label{thm:shrinkSel}
  If $y = S_\tau y_\tau$ for some $\tau \subset [m,n]$ and
  $y_k \neq 0$ for all $k\in \tau$ then
  \begin{align}\label{eq:shrinkSel0}
    \omega &= \Lambda V (S_\tau^T \Lambda V)^{-1} \ones,    
    \intertext{and}
    \label{eq:shrinkSel}
    \omega_i &= 
    1 - 
    \prod_{j \in \tau} \bigg( 1- \frac{\lambda_i}{\lambda_j} \bigg),
    &
    i &= 1,\ldots,m.
  \end{align}
  \begin{proof}
    Equation \eqref{eq:shrinkSel0} follows from the fact that
    $Y = S_\tau S_\tau^T Y$ and that $S_\tau^T V$ is
    non-singular. Then, $S_\tau^T \omega = \ones$, that is
    $\omega_j = 1$ when $j \in \tau$.  Now, from
    \eqref{eq:shrinkSel0},
    $1 - \omega_i = 1 - \sum_{k=1}^{n} \lambda_i^k \alpha_k$, where
    $\alpha_1,\ldots,\alpha_n$ are the elements of
    $\alpha = (S_\tau^T \Lambda V)^{-1} \ones$. That is,
    $1- \omega_i = p(\lambda_i)$ is the value of a polynomial $p$
    evaluated at $\lambda_i$.  That polynomial has degree $n+1$ and
    zeros at the points $\lambda_j$, $j \in \tau$ and value $1$ when
    evaluate at $0$, that is $p(0)=1$ and $p(\lambda_j) = 0$, for
    $j\in \tau$. The expression in \eqref{eq:shrinkSel} is a
    representation of that polynomial.
  \end{proof}
\end{lemma}

Note that, the above Lemma states that when the cardinality of $y$ is
exactly $n$, the shrinkage vector $w$ depends only on the sparsity
pattern of $y$ and not on the actual values of its elements. Then, the
following definition is well posed.
\begin{definition}\label{thm:defOmegaTau}
  For $\tau \in [m,n]$, $\omega_{(\tau)}$ denotes the vector
  of shrinkages corresponding to $y = S_\tau \ones$:
  \begin{align*}
    \omega_{(\tau)} = \Lambda V (S_\tau^T \Lambda V)^{-1} \ones.
  \end{align*}
\end{definition}

The following theorem states that any shrinkage vector is an average
of the degenerate shrinkage vectors $\omega_{(\tau)}$,
$\tau \in [m,n]$. It provides a novel representation, explicit on $y$,
for the shrinkage vector. The proof will be given in the next section.
\begin{theorem}\label{thm:omegaAvg}
  If the cardinality of $\psi = Y^2\ones$ is at least $n$, then
  \begin{align*}
    \omega = 
    \Big(    \sum_{\tau \in [m,n]} \psi^\tau \pi_\tau \Big)^{-1}
    \sum_{\tau \in [m,n]} \psi^\tau \pi_\tau \omega_{(\tau)},
  \end{align*}
  where
  \begin{align}\label{eq:pitau}
    \pi_\tau = 
    \lambda^\tau  
    \prod_{ \{j<i\} \subseteq \tau } (\lambda_i - \lambda_j)^2.
  \end{align}
\end{theorem}

That is, $\omega$ is an average (a convex combination) of the extreme
points $\omega_{(\tau)}$. The weights of that average, which are given
by $\psi^\tau \pi_\tau$, are multilinear functions of the squared
observations $\psi_1,\ldots,\psi_m$.  An analogous expression is
already known, but only for the case $n=m-1$ \cite{ButlerDenham:2000}.

Belonging $\omega$ to the convex hull of the set
$\{ \omega_{(\tau)} \;|\; \tau \in [m,n] \}$, extreme shrinkages will
arise when $y$ has (almost) cardinality $n$.  Furthermore, although
$\omega$ is a convex combination of the $\omega_{(\tau)}$s, the
mapping $y \to \omega$ does not range over the whole convex
hull. Indeed, as it will be shown later in Section~\ref{sec:geom}, the
range of $\omega$ is polyhedral (an union of simplicia) and not
necessarily convex.

Now, the set of possible shrinking/expanding patterns that can arise
in PLS regression is characterised.  Leaving $y$ unconstrained, that
set depends only on $m$, the number of eigenvalues and not on their
actual values.
\begin{theorem}\label{thm:shrinkPatterns}
  The following relations hold for $y$ and $\omega$:
  \begin{description}
  \item[a)] $\omega_m<1$ and the number of sign changes in $\omega -
    \ones$ is exactly $n$;
  \item[b)] for any signature with $n$ sign changes ending with a
     negative value, there's a value of $y$ such that $\omega - \ones$
    has that signature.
  \end{description}
\end{theorem}
The necessary part of Theorem~\ref{thm:shrinkPatterns} (point a) have
been proven in \cite{ButlerDenham:2000}. To the author's best
knowledge, the sufficient part (that is point b) is a new result.
As an example Table~\ref{tab:signs} reports the list of possible
signatures of $\omega - \ones$ for the case $m=6$ and $n=3$. A
positive sign indicates an expansion of the corresponding
coefficient. A negative one corresponds to a shrinkage of the
coefficient or a change in its sign, which may even be an expansion in
absolute value.
\begin{table}[ht]
  \caption{%
    Shrinkage patterns when $m=6$ and $n=3$. Positive 
    signs indicate expansions of the corresponding coefficient. 
  }
  \label{tab:signs}
  \footnotesize
  \begin{align*}
  \begin{array}{cc}
    \hline
    \sign(\omega - \ones) & \text{Positions of sign changes} \\
    \hline
    +-+--- & 1,2,3 \\
    +-++-- & 1,2,4 \\
    +-+++- & 1,2,5 \\
    +--+-- & 1,3,4 \\
    +--++- & 1,3,5 \\
    +---+- & 1,4,5 \\
    ++-+-- & 2,3,4 \\
    ++-++- & 2,3,5 \\
    ++--+- & 2,4,5 \\
    +++-+- & 3,4,5 \\
    \hline
  \end{array}    
  \end{align*}

\end{table}

\section{The geometry of PLS shrinkages}
\label{sec:geom}

To study the structure of the shrinkages it is convenient to work with
the quantity $z = \ones - \omega$. That vector contains the relative
residuals of PLS estimator, indeed
$z = Y^{-1} (y - \Lambda \hat{\beta})$, provided that $y$ does not
have null elements.
As it has already been noted on $\omega$, $z$ is a function of the
squared observations $\Psi = Y^2$ and it does not depend on the signs
of the elements of $y$. The object of this section is the study of the
mapping
\begin{align*}
  z : \mathcal{D} &\to \mathcal{I}_z ,
  &
  \psi &\to (I - Q(\psi)) \ones,
\end{align*}
where $Q(\psi) = \Lambda V (V^T \Psi \Lambda V)^{-1} V^T \Psi$.  and
$\mathcal{D}$ is the subset of $\Real_+^m$ whose elements have
cardinality not smaller than $n$.

Firstly, $\mathcal{I}_z$, the image of $z$, is a subset of the affine
space
\begin{align}\label{eq:Adef}
  \mathcal{A} = \ones +  \opspan(\Lambda V).
\end{align}
Next, $\omega = Q\ones$ is the projection on $\opspan(\Lambda V)$ of
$\ones$ along the null space of $\Psi V$. It turns out that $z$ lies
in the intersection of the null space of $\Psi V$ with the affine
space $\mathcal{A}$.  A sketch of that geometry is shown in
Figure~\ref{fig:1}.
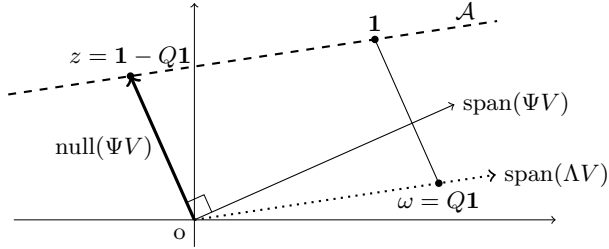
\begin{figure}[!h]
  \centering
  \begin{tikzpicture}[scale=2.4,every text node part/.style={font=\footnotesize}]
    \coordinate [label=-135:o] (o) at (0,0);
    \coordinate [label=above:$\ones$] (ones) at (1,1);
    \coordinate [label=right:$\opspan(\Lambda V)$] (u) at (1.67,.25);
    \coordinate [label=right:$\opspan(\Psi V)$] (v) at (1.44,.64);
    \coordinate (sa) at ($ (o)!3pt!(v) $);
    \coordinate (sb) at ($ (sa)!3pt!90:(v) $);
    \coordinate (sc) at ($ (o)!3pt!90:(v) $);
    \coordinate [label=below:{$\omega = Q\ones$}] (Qp) at (1.354,0.203);
    \coordinate [label=above:{$z = \ones - Q\ones$}] (z) at (-.354,.797);

    \path (o) edge[thick,dotted,->] (u);
    \path (o) edge[->] (v);
    \path (ones) edge[] (Qp);
    \path (o) edge[very thick,->] (z);
    \draw[thick,dashed] (-1.031,.695) -- (1.677,1.102);
    \draw[very thin] (sa) -- (sb) -- (sc);
    \draw[very thin,->] (-1,0) -- (2,0);
    \draw[very thin,->] (0,-.15) -- (0,1.2);

    \node at (-0.5,0.40) {$\opnull(\Psi V)$};
    \node at (1.50,1.16) {$\mathcal{A}$};

    \node [fill=black,shape=circle,inner sep=1pt] at (ones) {};
    \node [fill=black,shape=circle,inner sep=1pt] at (Qp) {};
    \node [fill=black,shape=circle,inner sep=1pt] at (z) {};
  \end{tikzpicture}

  \caption{\footnotesize \label{fig:1} Geometry of the oblique
    projection $Q(\psi)$. The dotted and dashed lines represent,
    respectively, $\opspan(\Lambda V)$ and
    $\ones + \opspan(\Lambda V)$. The thick segment corresponds to the
    vector $z$ and the other lines to $\opspan(\Psi V)$ and 
    $\ones + \opnull(\Psi V)$.}
\end{figure}

Then, the couple $(z,\psi)$ can be defined as any solution of the set
of constraints
\begin{align*}
  V^T \Psi z &= 0,
  &
  \psi &\in \mathcal{D},
  &
  z &\in \mathcal{A}.
\end{align*}

\begin{remark}\label{thm:zInverse}
  The mapping $z$ is not bijective, indeed its inverse maps
  $z \in \mathcal{I}_z$ to the set
  \begin{align*}
    \{ \psi \in \mathcal{D} \subseteq \Real_+^m \;|\; V^T Z \psi = 0 \}.
  \end{align*}
  Note that the closure of that set is the convex cone
  $\mathcal{C}_z = \{\psi \in \Real_+^m, \; V^T Z \psi=0 \}$ and,
  thus, can be characterised by means of a finite number of extremal
  rays.  This representation can exploited to derive the distribution
  density of $z$, or equivalently of $\omega$, which derives by
  integrating the density of $\psi$ over that set.
\end{remark}

Another quantity that will be used in this analysis is
$\alpha(\psi) = (V^T \Psi \Lambda V)^{-1} V^T \psi$. As it should have
been clear, the vectors $\alpha$, $\omega$ and $z$ are now considered
as functions of $\psi$. The mappings they define have the same domain,
images that will be denoted by $\mathcal{I}_\alpha$,
$\mathcal{I}_\omega$ and $\mathcal{I}_z$, respectively.  Note that,
since $\omega(\psi)$, $z(\psi)$ and $\alpha(\psi)$ are just a linear
or affine transformations of each other, most of the results derived
for any of these quantities apply to the other with only trivial
modifications. Working with $\alpha$ will correspond to working with
coordinates on the affine space $\mathcal{A}$, while working with $z$
has the advantage that its signature will have a precise geometric
meaning.

\subsection{The shrinkage vector is an average}

Consider now the behaviour of $\alpha$ and $z$ on the edge of
$\mathcal{D}$ where the cardinality is exactly $n$.
Analogously to Lemma~\ref{thm:shrinkSel}, when
$\psi = S_\tau \tilde{\psi}$, with $\tau \in [m,n]$ and
$\tilde\psi \in \Real^n_{++}$,
\begin{align}
  \label{eq:alphacorner}
  \alpha(\psi)  &= (S_\tau^T \Lambda V)^{-1} \ones,
  \intertext{and}
  \label{eq:zcorner}
  z(\psi)  &=  (I - \Lambda V (S_\tau^T \Lambda V)^{-1} S_\tau^T )\ones.
\end{align}
From Lemma~\ref{thm:shrinkSel} it follows that
$S_\tau^T \omega(\psi) = 1$ and $S_\tau^T z(\psi) = 0$.  As already
remarked, in these cases the value of $\alpha$ and $z$ depends only on
the sparsity pattern of $\psi$. The mappings
$\alpha_{(\cdot)}: [m,n] \to \Real^n$ and
$z_{(\cdot)}:[m,n] \to \Real^n$ are defined accordingly to
Definition~\ref{thm:defOmegaTau}. That is, for instance,
\begin{align}\label{eq:alphaTau}
  \alpha_{(\cdot)}:
  \tau &\to 
  \alpha_{(\tau)} = \alpha( S_\tau\ones ) =(S_\tau^T \Lambda V)^{-1}\ones.
\end{align}

The set of vectors $\{ \alpha_{(\tau)}, \; \tau \in [m,n]\}$ plays a
key role in the study of the image of $\alpha$. Indeed,
Theorem~\ref{thm:omegaAvg} is a corollary of the following result.
\begin{theorem}\label{thm:alphaAvg}
  For $\psi \in \mathcal{D}$,
  \begin{align*}
    \alpha(\psi) &= 
    \sum_{\tau \in [m,n]} p_\tau(\psi) \alpha_{(\tau)},
  \end{align*}
  where
  $p_\tau(\psi) = 
  \Big(\sum_{s \in [m,n]} \psi^s \pi_s \Big)^{-1} \psi^\tau \pi_\tau$,
  and $\pi_\tau$ is defined in \eqref{eq:pitau}.

  \begin{proof}
    Firstly note that $\alpha = \alpha(\psi)$ solves the equation
    \begin{align}\label{eq:alphaSys}
      V^T \Psi \Lambda V \alpha = V^T \Psi \ones.
    \end{align}
    The cardinality condition on $\psi$ is sufficient to guarantee the
    non-singularity of the coefficient matrix of \eqref{eq:alphaSys}.
    By the Kramer rule, the $k$-th element of $\alpha$ is given by
    \begin{align}\label{eq:alphaExpr}
      \alpha_k = 
      (-1)^k \frac{
        \det( V^T \Psi W_{(-k)} )
      }{
        \det( V^T \Psi \Lambda V )
      },
    \end{align}
    where $W_{(-k)}$ is the matrix obtained from the $m \times (n+1)$
    Vandermonde matrix $\big(\ones \; \Lambda V \big)$ after deleting the
    $(k+1)$-th column:
    \begin{align*}
      W_{(-k)} &= 
      \pmx{ 
        \ones & \Lambda \ones & \cdots & \Lambda^{k-1}\ones &
        \Lambda^{k+1}\ones & \cdots & \Lambda^n \ones 
      }.
    \end{align*}
    That is, $W_{(-k)}$ contains all the powers of $\lambda$ up to $n$
    excluding the $k$-th one.
    Note that $W_{(-0)} = \Lambda V$ so that 
    $\alpha_k = (-1)^k 
    \det( V^T \Psi W_{(-k)} ) / \det( V^T \Psi W_{(-0)} )$.
    
    Now, the Cauchy-Binet formula allows to express the determinants
    in \eqref{eq:alphaExpr} as
    \begin{align*}
      \det( V^T \Psi W_{(-k)} )
      &= 
      \sum_{\tau \in [m,n]} 
      \psi^\tau \det(S_{\tau}^TV) \det(S_\tau^T W_{(-k)}).
    \end{align*}
    Then, setting 
    $\pi_\tau = \det(S_\tau^TV) \det(S_\tau^T W_{(-0)})$
    gives
    \begin{align}\label{eq:alphak0}
      \alpha_k &= 
      \frac{1}{
        \displaystyle\sum_{\tau \in [m,n] } \psi^\tau \pi_\tau 
      }
      \sum_{\tau \in [m,n] } \psi^\tau \pi_\tau
      \alpha_{(\tau),k}, 
      &
      \text{with}&&
      \alpha_{(\tau),k} &= (-1)^k
      \frac{
        \det(S_\tau^T W_{(-k)})
      }{
        \det(S_\tau^T W_{(-0)})
      }.
    \end{align}
    The expression \eqref{eq:pitau} for $\pi_\tau$ derives by noting
    that, since $S_\tau^T V$ is a square Vandermonde matrix,
    \begin{align*}
      \det(S_\tau^TV) &= 
      \prod_{\{j<i\} \subseteq \tau}  (\lambda_i - \lambda_j),
      \intertext{and}
      \det(S_\tau^T W_{(-0)}) &= 
      \det(S_\tau^T \Lambda V) = \lambda^\tau \det(S_\tau^TV).
    \end{align*}
    The proof is concluded by noting that when $\psi$ has cardinality
    $n$, the summations in \eqref{eq:alphak0} have only one non-zero
    term and, thus, $\alpha_{(\tau),k}$ defined in \eqref{eq:alphak0}
    is equal to the $k$-th element of the vector $\alpha_{(\tau)}$
    defined in \eqref{eq:alphaTau}. Indeed,
  \end{proof}
\end{theorem}

\begin{corollary}\label{thm:avg}
  If $\psi \in \mathcal{D}$ then
  \begin{align*}
    \omega(\psi) &= 
    \sum_{\tau \in [m,n]} p_\tau(\psi) \omega_{(\tau)}
    &&\text{and}&
    z(\psi) &= 
    \sum_{\tau \in [m,n]} p_\tau(\psi) z_{(\tau)}.
  \end{align*}
  \begin{proof}
    The corollary follows directly noting that
    $\omega(\psi) = \Lambda V \alpha(\psi)$ and
    $z(\psi) = \ones - \Lambda V \alpha(\psi)$.
  \end{proof}
\end{corollary}

Since the object of the analysis will mainly be the relative residual
vector $z$ it is convenient, for future reference, to give the
explicit expression for $z_{(\tau),i}$ by rewriting
\eqref{eq:shrinkSel} as
\begin{align}\label{eq:zSel}
  z_{(\tau),i} &= \prod_{j \in \tau} 
  \Big(
  1 - \frac{\lambda_i}{\lambda_j}
  \Big),
  &
  \tau &\in [m,n].
\end{align}

The marginal dependence of $z$ on $\psi_k$ is characterised in
the following corollary.
\begin{corollary}\label{thm:alphaPsik}
  For $k \in [m]$, $z$ can be written in terms of $\psi_k$ as follows
  \begin{align*}
    z &= t \, z|_{\psi_k=0} + (1-t) (I - \lambda_k^{-1}\Lambda) z|^{(n-1)}_{\psi=\theta},
    &
      t &= \frac{1}{1+\psi_k g_k}, 
  \end{align*}
  with
  \begin{align*}
    g_k 
    &= 
      \Big(
      \sum\limits_{\tau \in \binom{[m]-\{k\}}{n}} 
      \psi^{\tau} \pi_\tau
      \Big)^{-1}
      \Big(
      \sum\limits_{\tau \in \binom{[m]-\{k\}}{n-1}}
      \psi^{\tau} \pi_{\tau \cup \{k\}}
      \Big),
  \end{align*}
  and where $z|_{\psi_k=0}$ is the values of $z$ obtained by setting
  $\psi_k=0$ and $z|_{\psi=\theta}^{(n-1)}$ is the value of $z$
  obtained at the previous step of the PLS method evaluated at the
  point $\theta = (I - \lambda_k^{-1}\Lambda)^2 \psi$. Note that
  $\theta_k = 0$.
  \begin{proof}
    From \eqref{eq:alphak0}, rewrite $\alpha$ as
    \begin{align*}
      z
      = 
      \frac{A + C \psi_k}{B + D \psi_k}
      = 
      \frac{A}{B} \cdot \frac{B}{B+D\psi_k}
      +
      \frac{C}{D} \cdot \frac{D\psi_k}{B+ D\psi_k},
    \end{align*}
    where
    \begin{align*}
    A 
    &= 
      \sum\limits_{\tau \in \binom{[m]-\{k\}}{n}} 
      \psi^\tau \pi_\tau z_{(\tau)},
      & 
        C
        &=
          \sum\limits_{\tau \in \binom{[m]-\{k\}}{n-1}} 
          \psi^\tau \pi_{\tau \cup \{k\}} z_{(\tau \cup \{k\})},
    \\
    B
    &= 
      \sum\limits_{\tau \in \binom{[m]-\{k\}}{n}} 
      \psi^\tau \pi_\tau,
      &
        D 
        &=
      \sum\limits_{\tau \in \binom{[m]-\{k\}}{n-1}} 
      \psi^\tau \pi_{\tau \cup \{k\}}.
  \end{align*}
  Now, write $\pi_{\tau \cup \{k\}}$ and $z_{(\tau \cup \{k\})}$ can be
  written in terms of $\pi_{\tau}$ and $z_{(\tau)}^{(n-1)}$ as follows
  \begin{align*}
    \pi_{\tau \cup \{k\}}
    &= 
      \prod_{j \in \tau \cup \{k\}} \lambda_j
      \prod_{ \{j<i\} \in \tau \cup \{k\}} (\lambda_i - \lambda_j)^2
      = 
      \lambda_k
      \prod_{j \in \tau} (\lambda_k - \lambda_j)^2
      \pi_{\tau}^{(n-1)}
  \end{align*}
  and  %
  $
  z_{(\tau \cup \{k\}),p}
  = 
  \prod_{j \in \tau \cup \{k\}} 
  ( 1 - \lambda_j^{-1}\lambda_p )
  = 
  ( 1 - \lambda_k^{-1}\lambda_p )
  z_{(\tau),p}^{(n-1)}
  $,
  that is,
  $z_{(\tau \cup \{k\})} = ( I  - \lambda_k^{-1}\Lambda ) z_{(\tau)}^{(n-1)}$.
  Here $z_{(\tau)}^{(n-1)}$ is the value of $z_{(\tau)}$ at the step
  $n-1$.  It follows that
  \begin{align*}
    \frac{C}{D} 
    &=
      (I - \lambda_k^{-1}\Lambda)
      \frac{
      \displaystyle%
      \sum\limits_{\tau \in \binom{[m]-\{k\}}{n-1}} 
      \psi^\tau \pi_{\tau}^{(n-1)} 
      z_{(\tau)}^{(n-1)} 
      \prod_{j\in \tau} (1 - \lambda_k^{-1}\lambda_j)^2
      }{
      \displaystyle\sum\limits_{\tau \in \binom{[m]-\{k\}}{n-1}} 
      \psi^\tau \pi_{\tau}^{(n-1)}
      \prod_{j\in \tau} (1 - \lambda_k^{-1}\lambda_j)^2
      }
    \\
    &=
      (I - \lambda_k^{-1}\Lambda)
      \frac{
      \displaystyle%
      \sum\limits_{\tau \in \binom{[m]-\{k\}}{n-1}} 
      \theta^\tau \pi_{\tau}^{(n-1)} 
      z_{(\tau)}^{(n-1)} 
      }{
      \displaystyle\sum\limits_{\tau \in \binom{[m]-\{k\}}{n-1}} 
      \theta^\tau \pi_{\tau}^{(n-1)}
      },
  \end{align*}
  where $\theta_j = (1 - \lambda_k^{-1}\lambda_j)^2 \psi_j$,
  $j=1,\ldots,m$.  The vector $\theta$ can also be written as
  $\theta = (I - \lambda_k^{-1}\Lambda)^2\psi$ and then
  $C/D = (I - \lambda_k^{-1}\Lambda) z|^{(n-1)}_{\psi=\theta}$.
\end{proof}
\end{corollary}

Corollary~\ref{thm:alphaPsik} shows that, by varying the value of
$\psi_k$, the vector $z$ moves along the segment with extremes
$z|_{\psi_k=0}$ and
$(I - \lambda_k^{-1}\Lambda)z|^{(n-1)}_{\psi=\theta}$. Those are
reached when $\psi_k =0$ or $\psi_k=\infty$, respectively.  
That corollary gives also a representation of the shrinkages obtained
at the $(n-1)$-th step as a limit of the ones obtained at the $n$-th
step.  That is, when $\psi_k$ is large, the shrinkage factors can be
approximated rescaling those obtained from $(n-1)$-steps of the PLS
procedure applied to model where $\psi$ is replaced by $\theta$:
$z = (I - \lambda_k^{-1}\Lambda)z|_{\psi=\theta}^{(n-1)}$ for large
values of $\psi_k$.
The above result allows also to derive the shrinkage or the estimator
distribution conditional on
$\psi_1,\ldots,\psi_{k-1}, \psi_{k+1},\ldots,\psi_m$.

\subsection{The shrinkage's range}

Consider now the characterisation of $\mathcal{I}_\alpha$, that is the
range of the mapping $\alpha$.  Theorem~\ref{thm:alphaAvg} states that
$\mathcal{I}_\alpha 
\subseteq \conv( \{\alpha_{\tau} \;|\; \tau \in [m,n] \} )$,
where $\conv(X)$ denotes the convex hull of the set $X$.  
In order to derive the actual shape of $\mathcal{I}_\alpha$ the two
auxiliary results are introduced.  The first one is given in the
following lemma which states that the above inclusion becomes an
equality, when $m = n+1$.
\begin{lemma}\label{thm:domEqConv}
  When $m=n+1$, 
  $\cl(\mathcal{I}_\alpha) =
  \conv( \{\alpha_{\tau}, \; \tau \in [m,n] \} )$,
  where $\cl(X)$ denotes the closure of the set $X$.

  \begin{proof}
    It need only to be proven that any element of the righthand side
    can be written as $\alpha(\psi)$ for some $\psi \in \Real_+^m$.
    To this end, firstly note that 
    \begin{align*}
      [m,n] = 
      \big\{ \tau_i = [m] - \{i\}, \; i = 1,\ldots,m \big\}.
    \end{align*}
    Then, consider a generic $\alpha_*$ in the convex hull of 
    $\{\alpha_{(\tau_1)}, \ldots, \alpha_{(\tau_m)}\}$:
    \begin{align*}
      \alpha_* &= \sum_{i=1}^m c_i \alpha_{(\tau_i)},
    \end{align*}
    where $c_i>0$ and $\sum_i c_i = 1$.  By Theorem~\ref{thm:alphaAvg},
    it is sufficient to find a $\psi \in \Real_+^m$ such that $c_i =
    p_{\tau_i}$. To this end, note that, since $\psi^{\tau_i} =
    \tau_i^{-1} (\tau_1 \cdots \tau_m)$,
    \begin{align*}
      p_{\tau_i} &= 
      \frac{
        \psi_i^{-1} \pi_{\tau_i}
      }{
        \sum_{j=1}^m\psi_j^{-1} \pi_{\tau_j}
      }.
    \end{align*}
    Then, choosing $\psi_i = \kappa \pi_{\tau_i}/c_i$, $i=1,\ldots,m$,
    leads to $p_{\tau_i} = c_i$.  Now, since $\pi_{\tau_i}>0$, the
    elements of $\psi$ are positive provided that $c_i>0$. Finally,
    taking the closure of $\mathcal{I}_\alpha$ allows to include the
    remaining points of the convex hull.
  \end{proof}
\end{lemma}

The second result is given in the next lemma, where it is shown that
any element of $\mathcal{I}_\alpha$ can be written as $\alpha(\psi)$,
with the cardinality of $\psi$ being $n+1$.  For simplicity and
without loss of generality that result is derived in term of $z$
instead of $\alpha$.

\begin{lemma}\label{thm:reduce}
  Let $z_\star \in \mathcal{I}_z$, there exists $\psi \in \Real^n$, $\psi
  \geq 0$, with cardinality $n+1$ such that $z_\star = z(\psi)$.

  \begin{proof}
    Consider a vector $\psi_\star \in \mathcal{D}$ such that $z_\star
    = z(\psi)$ and $\ones^T \psi_\star = 0$.  Note that, by the
    homogeneity of $z$, this normalisation can be imposed without
    losing in generality.
    Now, as remarked in Section~\ref{sec:geom}, such a vector should
    satisfy
    \begin{align*}
      V^T Z_\star \psi_\star = 0, \quad \ones^T \psi_\star = 0, \quad \psi_\star \geq 0
    \end{align*}
    That equation states that the origin belong to the convex-hull of
    the columns of $Z_\star V$.  Then, by the Minkowski-Carath\`eodory
    theorem, it is also contained into the convex hull of a subset of
    this set of vectors having cardinality $n+1$. That is,
    \begin{align*}
      0 &= V^TZ_\star \psi, & 
      \psi \geq 0,\quad
      \ones^T \psi = 1,\quad
      \card(\psi) = n+1,
    \end{align*}
    which implies also that $\psi \in \mathcal{D}$ and thus $z(\psi) = z_\star$.
  \end{proof}
\end{lemma}

Now, a precise description of $\mathcal{I}_z$ and, in turns of
$\mathcal{I}_\alpha$, follows directly as a corollary of Lemmata
\ref{thm:domEqConv} and~\ref{thm:reduce}.
\begin{corollary}
  \begin{align*}
      \cl(\mathcal{I}_z) &= 
      \bigcup_{T \in [m,n+1]} 
      \conv\!\Big( \Big\{ z_{(\tau)}, \; \tau \in \binom{T}{n} \Big\} \Big).
  \end{align*}
\end{corollary}
Note that, in general, the image of $z$ does not correspond to the
convex hull of the points $z_{(\tau)}$, $\tau \in
[m,n]$.
However, to characterise its geometry it is sufficient to study the
relations between the tetrahedra with vertices
$\{ z_{(\tau)}, \tau \in \binom{T}{n} \}$, $T \in [m,n+1]$.
The following proposition states that all the points $z_{(\tau)}$ that
share a common subset $\tau_0$ are aligned in a subspace of dimensions
$n-k$. In particular, when $k=n-1$ these points are aligned (see also
Corollary~\ref{thm:alphaPsik}).
\begin{proposition}
  Let $k<n$ and $\tau_0 \in [m,k]$. Then, any $z_{(\tau)}$ with
  $\tau \in [m,n]$ and $\tau_0 \subset \tau$ belongs to the
  subspace of $\Real^m$
  \begin{align*}
    \mathcal{Z}_{\tau_0} = 
    \{ 
    z \subset \Real^m \;|\; z_i = 0, \text{ for } i \in \tau_0 
    \}.
  \end{align*}

  \begin{proof}
    Let $S_\tau$ be the selection matrix corresponding to $\tau$.  To
    prove this result it is sufficient to note that $S_{\tau}^T z_{(\tau)}
    = 0$ and that $\tau_0 \subset \tau$.
  \end{proof}
\end{proposition}

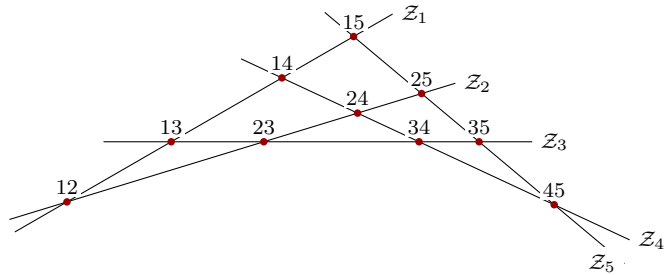
\begin{figure}[b]
  \centering

\begin{tikzpicture}[
  po/.style={fill=red!60!black,shape=circle,inner sep=1pt},
  pl/.style={above,font=\scriptsize,fill=white,inner sep=1pt,outer sep=2pt},
  li/.style={fill=white,font=\scriptsize},
  si/.style={font=\tiny,fill=white,inner sep=1pt}
  ]

  \draw[name path=one]  (30:0)   ++(30:0)    -- ++(30:5.8)  node[li,right] {$\mathcal{Z}_1$};
  \draw[name path=two]  (30:.8)  ++(17:-.8)  -- ++(17:6.2)  node[li,right] {$\mathcal{Z}_2$};
  \draw[name path=three](30:2.4) ++(0:-.9)   -- ++(0:5.7)   node[li,right] {$\mathcal{Z}_3$};
  \draw[name path=four] (30:4.1) ++(-25:-.6) -- ++(-25:5.7) node[li,right] {$\mathcal{Z}_4$};
  \draw[name path=five] (30:5.2) ++(-40:-.5) -- ++(-40:5.2) node[li,left] {$\mathcal{Z}_5$};

  \node [name intersections={of=one and two, by=ab},po] at (ab) {};
  \node [name intersections={of=one and three, by=ac},po] at (ac) {};
  \node [name intersections={of=one and four, by=ad},po] at (ad) {};
  \node [name intersections={of=one and five, by=ae},po] at (ae) {};

  \node [name intersections={of=two and three, by=bc},po] at (bc) {};
  \node [name intersections={of=two and four, by=bd},po] at (bd) {};
  \node [name intersections={of=two and five, by=be},po] at (be) {};

  \node [name intersections={of=three and four, by=cd},po] at (cd) {};
  \node [name intersections={of=three and five, by=ce},po] at (ce) {};

  \node [name intersections={of=four and five, by=de},po] at (de) {};

  \path node[pl] at (ab) {$12$};
  \path node[pl] at (ac) {$13$}  node[pl] at (bc) {$23$};
  \path node[pl] at (ad) {$14$}  node[pl] at (bd) {$24$} node[pl] at (cd) {$34$};
  \path node[pl] at (ae) {$15$}  node[pl] at (be) {$25$} node[pl] at (ce) {$35$};

  \path node[pl] at (de) {$45$};

\end{tikzpicture}
  \medskip

  \caption{\footnotesize %
    Geometry of $z_{(\tau)}$ for $m=5,n=2$.  Each line represents an
    affine space $\mathcal{Z}_i$.}
  \label{fig:aGeom2}

\end{figure}
\begin{figure}[htb]
   \centering

\tdplotsetmaincoords{30}{40}
\begin{tikzpicture}[baseline=(current bounding box.north),
  scale=1,
  tdplot_main_coords,
  decoration={random steps,segment length=5pt,amplitude=2pt},
  po/.style={fill=black,shape=circle,inner sep=.7pt},
  pl/.style={above,font=\scriptsize,fill=white,inner sep=1pt,outer sep=2pt},
  li/.style={shape=rectangle,fill=white,draw,inner sep=1.5pt,outer sep=2pt,font=\scriptsize},
  hl/.style={thin,dotted}
  ]

  \path 
      coordinate (abc) at (3,0,0) {}
      coordinate (abf) at (0,6,-1) {}
      coordinate (aef) at (8,6,0) {}
      coordinate (def) at (8,6,12) {};
  \path (abc) -- (abf) 
      coordinate[pos=0.2] (abd) {} coordinate[pos=0.6] (abe) {};
  \path (abf) -- (aef)
      coordinate[pos=0.18] (acf) {} coordinate[pos=0.6] (adf) {};
  \path (aef) -- (def) 
      coordinate[pos=0.42] (bef) {} coordinate[pos=0.73] (cef) {};

  \draw[hl,name path=ab] (abc) -- (abf);
  \draw[hl,name path=ac] (abc) -- (acf);
  \draw[hl,name path=ad] (abd) -- (adf);
  \draw[hl,name path=ae] (abe) -- (aef);
  \draw[hl,name path=af] (abf) -- (aef);
  \draw[hl,name path=bf] (abf) -- (bef);
  \draw[hl,name path=cf] (acf) -- (cef);
  \draw[hl,name path=df] (adf) -- (def);
  \draw[hl,name path=ef] (aef) -- (def);
  
  \node[name intersections={of=ac and ad, by=acd}] at (acd) {};
  \node[name intersections={of=ac and ae, by=ace}] at (ace) {};
  \node[name intersections={of=ad and ae, by=ade}] at (ade) {};
  \node[name intersections={of=bf and cf, by=bcf}] at (bcf) {};
  \node[name intersections={of=bf and df, by=bdf}] at (bdf) {};
  \node[name intersections={of=cf and df, by=cdf}] at (cdf) {};

  \draw[hl,name path=bc] (abc) -- (bcf);
  \draw[hl,name path=bd] (abd) -- (bdf);
  \draw[hl,name path=be] (abe) -- (bef);
  \draw[hl,name path=cd] (acd) -- (cdf);
  \draw[hl,name path=de] (ade) -- (def);

  \node[name intersections={of=bc and bd, by=bcd}] at (bcd) {};
  \node[name intersections={of=bd and be, by=bde}] at (bde) {};
  \node[name intersections={of=bc and be, by=bce}] at (bce) {};
  \node[name intersections={of=cd and de, by=cde}] at (cde) {};

  \draw[hl] (ace) -- (bce);

  \fill[fill=orange,opacity=.4] (acd) -- (acf) -- (cdf) -- (acd);
  \fill[fill=yellow,opacity=.4] (acd) -- (adf) -- (cdf) -- (acd);
  \fill[fill=red,opacity=.4] (abc) -- (bcf) -- (abf) -- (abc);
  \fill[fill=brown,opacity=.4] (ade) -- (aef) -- (def) -- (ade);

  \draw[hl,dashed] (abd) -- (bde) -- (bef);
  \draw[hl,dashed] (abe) -- (bce) -- (cef);

  \draw[thick] (abc) -- (abf) -- (bcf) -- (cdf);
  \draw[thick] (def) -- (aef) -- (ade) -- (acd);
  \draw[thick,dotted] (abc) -- (acd);
  \draw[thick,dotted] (cdf) -- (def);
  \draw[thick] (abc) -- (bcd) -- (cde) -- (def);
  \draw[thick,dotted] (abf) -- (aef);
  \draw[thin] (bcd) -- (bcf);
  \draw[thin] (acd) -- (bcd);
  \draw[thin] (ade) -- (cde);
  \draw[thin,dotted] (cde) -- (cdf);

  \path[pl] 
    node at (abc) {$123$} node at (abd) {$124$}
    node at (abe) {$125$} node at (abf) {$126$}
    node at (acd) {$134$} node at (ace) {$135$}
    node at (acf) {$136$} node at (ade) {$145$}
    node at (adf) {$146$} node at (aef) {$156$};
  \path[pl]
    node at (bcd) {$234$} node at (bce) {$235$}
    node at (bcf) {$236$} node at (bde) {$245$}
    node at (bdf) {$246$} node at (bef) {$256$};
  \path[pl]
    node at (cde) {$345$} node at (cdf) {$346$}
    node at (cef) {$356$} node at (def) {$456$};
\end{tikzpicture}
\qquad
\begin{tikzpicture}[
  baseline=(current bounding box.north),
  leg/.style={minimum size=1ex,draw=black,fill opacity=.4},
  legText/.style={font=\footnotesize}
  ]
  \node[leg,fill=red,   label={[legText]right:{$\mathcal{Z}_2$}}] (r) at (0,0) {}; 
  \node[leg,fill=orange,label={[legText]right:{$\mathcal{Z}_3$}}] (o) at (0,-2em) {}; 
  \node[leg,fill=yellow,label={[legText]right:{$\mathcal{Z}_4$}}] (y) at (0,-4em) {}; 
  \node[leg,fill=brown, label={[legText]right:{$\mathcal{Z}_5$}}] (b) at (0,-6em) {}; 
\end{tikzpicture}
  
\caption{\footnotesize %
  Geometry of $z_{(\tau)}$ for $m=6$, $n=3$ projected on the affine
  space $\mathcal{A}$.  Each affine space $\mathcal{Z}_i$ is shown in
  a different color except for $\mathcal{Z}_1$ and $\mathcal{Z}_6$
  which are not visible here.  Lines correspond to the affine spaces
  $\mathcal{Z}_i \cap \mathcal{Z}_j$, $i \neq j$. }
  \label{fig:aGeom3}
\end{figure}
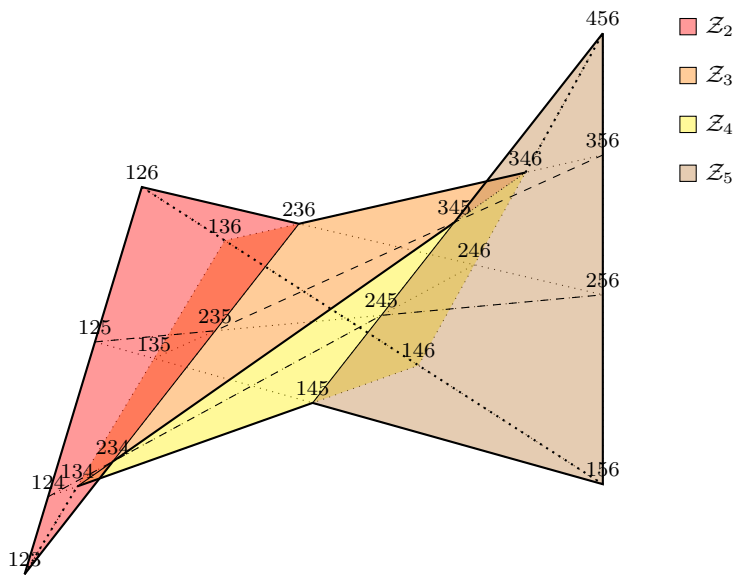
When $n-k=1$, besides being aligned, the points $z_{(\tau)}$,
$\tau \supset \tau_0$, follow also a specific ordering. Indeed, let
$\tau = \tau_0 \cup \{ k \}$, by \eqref{eq:zSel} the elements of
$z_{(\tau)}$ can be written as
$z_{(\tau),i} = 
  ( 1 - \lambda_i/\lambda_k )
  \prod_{j \in \tau_0} ( 1 - \lambda_i/\lambda_j )$,
$i=1,\ldots,n$.
For example, when $n=3$, the points 
$z_{(2\underline{3}7)},
 z_{(2\underline{4}7)},
 z_{(2\underline{5}7)}$ and 
$z_{(2\underline{6}7)}$ are
all aligned and ordered, in the sense that 
$z_{(2\underline{4}7)}$ and
$z_{(2\underline{5}7)}$ belong to the segment
$z_{(2\underline{3}7)}-z_{(2\underline{6}7)}$ and $z_{(2\underline{5}7)}$ to
the segment $z_{(2\underline{4}7)}-z_{(2\underline{6}7)}$. 
It also implies that when $n=3$ the tetrahedron
$\mathcal{T}_{1234} = \conv\{ z_{(\tau)} \;|\; \tau \subset
\{1,2,3,4\}, \#\tau =3 \}$
is contained into the tetrahedron
$\mathcal{T}_{1236} = \conv\{ z_{(\tau)} \;|\; \tau \subset
\{1,2,3,6\}, \#\tau =3 \}$.
Two examples of the possible configurations of $z_{(\tau)}$ are
reported in Figures~\ref{fig:aGeom2} and~\ref{fig:aGeom3} for the
cases $n=2$ and $n=3$, respectively. It should be remarked that, even
though proportions changes with $\lambda$, the main features of the
geometry, that is spatial alignment and ordering, remain fixed.

Each linear space $\mathcal{Z}_i$, $i=1,\ldots,m$, defines a partition
of $\Real^m$ into three sets
$\mathcal{Z}_i^\pm = \{ z \;| \; \pm z_i > 0 \}$ and
$\mathcal{Z}_i$. Any signature of $z$ is associated to a specific
intersection of these sets. Figures~\ref{fig:aGeom2} and
\ref{fig:aGeom3} anticipated the geometry of the possible
configurations of the points $z_{(\tau)}$, $\tau \in [m,n]$
which ultimately depends on of possible signatures of the elements of
$\mathcal{I}_z$. This is the aspect that will be addressed in the next
subsection.

\subsection{The signature of $z$}

Any subset $\tau \in[m,n]$ with cardinality $n$ corresponds to a
single point $z_{(\tau)} \in \mathcal{I}_z$. Now, consider a subset
$T$ of length $n+1$. That set contains $n+1$ subsets of length $n$,
each one associated to a point of $\mathcal{I}_z$. It is not difficult
to show that these points are linear independent, and thus their
convex hull is a tetrahedron or, more precisely, an $n$-dimensional
simplex of $\Real^m$. The following definition introduces the notation
used to indicate that set.
\begin{definition}
  For $T \in [m,n+1]$, $\mathcal{T}_T$ denotes the
  $n$-dimensional simplex (a tetrahedron when $n=3$)
  \begin{align*}
    \mathcal{T}_T = 
    \conv\Big\{ z_{(\tau)}, \;|\;  \tau \in \binom{T}{n}  \Big\}.
  \end{align*}
\end{definition}

\begin{example}
  $\mathcal{T}_{12357}$ is the simplex having vertices $z_{(2357)}$,
  $z_{(1357)}$, $z_{(1257)}$, $z_{(1237)}$ and $z_{(1235)}$. This
  simplex has $n=5$ faces. Each of them is defined by a $n$ vertices
  and belong to one of the linear spaces
  $\mathcal{Z}_1,\ldots,\mathcal{Z}_5$. For instance, the corner
  points $z_{(1357)}$, $z_{(1257)}$, $z_{(1237)}$ and $z_{(1235)}$
  identify the face which belong to $\mathcal{Z}_1$. The following
  proposition will show that $\mathcal{T}_{12345} = (\mathcal{Z}_1^+
  \cup \mathcal{Z}_2^- \cup \mathcal{Z}_3^+ \cup \mathcal{Z}_4^- \cup
  \mathcal{Z}_5^+) \cap \mathcal{A}$. Recall that $\mathcal{A}$ is the
  $n$-dimensional affine space defined in \eqref{eq:Adef}.
\end{example}

Now, the signature of the elements in these simplices, $z \in
\mathcal{T}_T$, is considered.  It turns out that the signature at
some specific positions is fixed.
\begin{proposition}\label{thm:tetraSign}
  Let $z \in \mathcal{T}_t$ with $t = \{ t_0 < t_1 < \cdots < t_n\}$, 
  $t_0>0$ and $t_n\leq m$. Then,
  \begin{align*}
    (-1)^{n-k} z_{t_k} &\geq 0, 
    &
    k &= 1,\ldots,n,
    \\
    z_i &\geq 0, & i&>t_n,
    \\
    (-1)^nz_i &\geq 0, & i&<t_0.
  \end{align*}

  \begin{proof}
    Consider the $t_i$-th element of a generic corner point
    $z_{(\tau)}$, $\tau \in \binom{t}{n}$. Then, $z_{(\tau),i} = 0$ if
    $i\in \tau$ otherwise, by \eqref{eq:zSel},
    $(-1)^{n+1-i}z_{(\tau),i} >0$.  To conclude the proof it is
    sufficient to note that $z$ is a convex combination of the corner
    points.
  \end{proof}
\end{proposition}

\begin{proposition}\label{thm:tetraSignUndet}
  Let $t = a \cup b \in [m,n+1]$ and $i \in [m]$. 
  If 
  \begin{align*}
    \max a < i <\min b,
  \end{align*}
  then $\mathcal{Z}_i$ crosses the interior of the simplex
  $\mathcal{T}_t$. 
  That is, there exists  $z^+,z^-,z^0 \in \mathcal{T}_t$ such that
  $z^+_i>0$, $z^-_i<0$ and $z^0_i= 0$.
  \begin{proof}
    Choose $j \in a$ and $k \in b$, and set $s = T - \{j\}$ and
    $t = T- \{k\}$.  The $i$-th elements of $z_{(s)}$ and $z_{(t)}$
    have different signs. Indeed, by \eqref{eq:zSel},
    \begin{align*}
      z_{(s),i}z_{(t),i} 
      &= 
      \prod_{h \in s}\Big( 1 - \frac{\lambda_i}{\lambda_j} \Big)
      \prod_{h \in t}\Big( 1 - \frac{\lambda_i}{\lambda_k} \Big)
      \\
      &= 
      \Big( 1 - \frac{\lambda_i}{\lambda_j} \Big)
      \Big( 1 - \frac{\lambda_i}{\lambda_k} \Big)
      \prod_{h \in s \cap t}\Big( 1 - \frac{\lambda_i}{\lambda_h}
      \Big)^2
      < 0,
    \end{align*}
    since $\lambda_j > \lambda_i > \lambda_k$. 
  \end{proof}
\end{proposition}

Propositions~\ref{thm:tetraSign} and~\ref{thm:tetraSignUndet}
establishes a pattern for the signature of the elements of a simplex
$\mathcal{T}_t$, $t \in [m,n+1]$. The following example gives
an illustration of the above results.

\begin{example}\label{ex:1}
  Let $m=12$, $n=4$, $t= \{ 2,5,7,8,10 \}$. Then, the signature of the
  corner elements of $\mathcal{T}_t$ and of a generic element $z$ is
  reported in Table~\ref{tab:exBis}, where $+$, $-$, $0$ and $\times$
  denote non-negative, non-positive, null and unconstrained elements.
  \begin{table}[h]
    \caption{
      Signatures of the vertices of
      $\mathcal{T}_{\{2,5,7,8,10\}}$. The last row shows the sign pattern 
      of an internal point as prescribed by
      Propositions~\ref{thm:tetraSign} 
      and~\ref{thm:tetraSignUndet}.
    }
    \label{tab:exBis}

    \begin{align*}
      {
        \setlength{\arraycolsep}{2pt}
        \begin{array}{lc*{12}{c}}
      \hline
                     &&  & & & & & & & & &1&1&1 \\[-0.5ex]
                     && 1&2&3&4&5&6&7&8&9&0&1&2 \\
      \hline
      z_{(5,7,8,10)}   && +&+&+&+&0&-&0&0&-&0&+&+ \\
      z_{(2,7,8,10)}   && +&0&-&-&-&-&0&0&-&0&+&+ \\
      z_{(2,5,8,10)}   && +&0&-&-&0&+&+&0&-&0&+&+ \\
      z_{(2,5,7,10)}   && +&0&-&-&0&+&0&-&-&0&+&+ \\
      z_{(2,5,7,8) }   && +&0&-&-&0&+&0&0&+&+&+&+ \\
      \hline
      z   && +&+&\times& \times &-& \times &+&-& \times&+&+&+ \\
      \hline
         \end{array}
       }
    \end{align*}

  \end{table}
\end{example}

Proposition~\ref{thm:tetraSignUndet} only characterise the marginal
structure of the signature of $z$. There is, however, a conjoint
structure that links consecutive undetermined signs.
The investigation of that conjoint structure is here tackled by using
tools related to the concept of total positivity of matrices
\cite{FallatJohnson:2011}. Before proceeding with the proof, it is
necessary to introduce some definitions and results on that subject.
\begin{definition}[Total positivity]
  A matrix $A \in \Real^{m \times n}$ is totally non-negative
  (positive) if the determinant of any square submatrix of $A$ is
  non-negative (positive).
\end{definition}

Applying a totally non-negative (positive) matrix to a vector cannot
increase the number of sign variations. This number is not uniquely
identified when some of the elements are null and maximal/minimal
quantities should be used.
\begin{definition}
  Let $x \in \Real^n$ be a vector having cardinality $n$ (no zero
  elements), then $v(x)$ denotes the number of sign changes in $x$:
  \begin{align*}
    v(x) &= \# \{ x_{i}x_{i+1}<0, \; 1 < i \leq n \}.
  \end{align*}
\end{definition}
\begin{definition}\label{def:vMinMax}
  For $x\in \Real^n$,
  \begin{align*}
    v_m(x) &= \min \{ 
    v(y) \;|\; 
    \card(y) = n \text{ and }  x_iy_i \geq 0,\; 1 \leq i\leq n
    \},
    \\
    v_M(x) &= \max \{ 
    v(y) \;|\; 
    \card(y) = n \text{ and }  x_iy_i \geq 0,\; 1 \leq i\leq n
    \},
  \end{align*}
  where $x \in \Real^n$.
\end{definition}
The max (min) used in Definition~\ref{def:vMinMax} is computed among
the vectors $y$ whose signature is concordant with that of the
non-null part of $x$.
That is, $v_m(x)$ counts the number of sign changes once the null
elements have been removed from $x$. Let $\bar{x}$ be obtained from
$x$, by replacing the zero elements with signed values, then
$v(\bar{x}) \in [v_m(x), v_M(x)]$. Then, the definition of $v(x)$
could be extended to the mapping $x \to [v_m(x),v_M(x)] \cap
\mathbb{N}$.

These measures of sign changes satisfy the property
\begin{align*}
  v_m(x) &\leq v_M(x) \leq n-1,
  &
  x &\in \Real^n.
\end{align*}
As stated above, totally non-negative (positive) matrices do not increase
the number of sign changes when applied to a vector. That fact is 
more precisely stated in the following Theorem.
\begin{theorem}\label{thm:TPchanges}
  Let $A: m \times n$ and $x \in \Real^n$.  %
  If $A$ is totally non-negative then $v_m(Ax) \leq v_m(x)$.  %
  If $A$ is totally positive and $x \neq 0$, then
  \begin{align*}
    v_M(Ax) &\leq v_m(x).
  \end{align*}
  \begin{proof}
    See Theorems 4.2.2 and 4.3.5 in \cite{FallatJohnson:2011}.
  \end{proof}
\end{theorem}

It is possible now, to establish the number of sign variations of a
generic $z \in \mathcal{Z}$. The following Lemma and Corollary
complete the characterisation given in Propositions
\ref{thm:tetraSign} and~\ref{thm:tetraSignUndet}.
\begin{lemma}\label{thm:zsignature}
  Assume that $n<m$. Then,
  \begin{align}\label{eq:vM}
    v_M(z) \leq n.
  \end{align}
  If $\card(\psi)>n$, then
  \begin{align}\label{eq:vm0}
    v_m(z) &\in \{ n-1,n\}.
  \end{align}
  and $v_m(z) = n$ when $\card(\psi)>n$, $z_1 \neq 0$ and $z_m
  \neq 0$.
\end{lemma}

\begin{corollary}
  Assuming $n<m$, if $\card(z) = m$ then $v(z) = n$, $z_m>0$ and $(-1)^nz_1>0$.
\end{corollary}

Before going to the proof of the Lemma, an example of its application
to the characterisation of the signatures of the elements of
$\mathcal{T}_t$ is given.

\begin{example}[Continuation of Example~\ref{ex:1}]
  The above corollary allows to identify the signature of elements of
  $\mathcal{T}_{\{2,5,7,8,10\}}$. For example, the uncertainty at the positions 3 and
  4 corresponds to only one change of sign that can occur before the
  3rd, 4th or 5th position.
  The possible signatures are reported in Table~\ref{tab:exTer}.
  \begin{table}[tb]
    \caption{%
      Signatures of the elements of
      $\mathcal{T}_{\{2,5,7,8,10\}}$. The first column contains the positions
      of sign changes and the first row shows the pattern
      prescribed by Propositions~\ref{thm:tetraSign} 
      and~\ref{thm:tetraSignUndet}.
    }
    \label{tab:exTer}
    \begin{align*}
      {
        \setlength{\arraycolsep}{2pt}
        \begin{array}{lc*{12}{c}}
          \hline
          &&  & & & & & & & & &1&1&1 \\[-.4ex]
          Pos. && 1&2&3&4&5&6&7&8&9&0&1&2 \\
          \hline
          && +&+&\times& \times &-& \times &+&-& \times&+&+&+ \\
          \hline\\[-2ex]
          5,7,8,10  && +&+&+&+&-&-&+&-&-&+&+&+ \\[-.3ex]
          4,7,8,10  && +&+&+&-&-&-&+&-&-&+&+&+ \\[-.3ex]
          3,7,8,10  && +&+&-&-&-&-&+&-&-&+&+&+ \\[0.5ex]
          5,6,8,10  && +&+&+&+&-&+&+&-&-&+&+&+ \\[-.3ex]
          4,6,8,10  && +&+&+&-&-&+&+&-&-&+&+&+ \\[-.3ex]
          3,6,8,10  && +&+&-&-&-&+&+&-&-&+&+&+ \\[0.5ex]
          5,6,8,9   && +&+&+&+&-&-&+&-&+&+&+&+ \\[-.3ex]
          4,6,8,9   && +&+&+&-&-&-&+&-&+&+&+&+ \\[-.3ex]
          3,6,8,9   && +&+&-&-&-&-&+&-&+&+&+&+ \\[0.5ex]
          5,7,8,9   && +&+&+&+&-&+&+&-&+&+&+&+ \\[-.3ex]
          4,7,8,9   && +&+&+&-&-&+&+&-&+&+&+&+ \\[-.3ex]
          3,7,8,9   && +&+&-&-&-&+&+&-&+&+&+&+ \\
          \hline
        \end{array}
      }
    \end{align*}
  \end{table}

\end{example}

The following two results are instrumental to the proof of Lemma
\ref{thm:zsignature}.
\begin{lemma}\label{thm:VPos}
  The Vandermonde matrices $V$ and $\Lambda V$ are totally positive.
  \begin{proof}
    It is a consequence of the ordering assumed for $\lambda$ (see the
    introductory chapter of \cite{FallatJohnson:2011}).
  \end{proof}
\end{lemma}
\begin{proposition}\label{thm:phiznull}
  Let $\Psi \in \mathcal{D}$, if $\Psi z(\psi) = 0$ then
  $\card(\psi)=n$.
  \begin{proof}
    Firstly note that, since $z = \ones - \Lambda V\alpha$ and
    $\Lambda V$ has full column rank, the vector $z$ cannot have more
    than $n$ zero elements. Besides, $\mathcal{D}$ does not contains
    vectors with more than $n$ non-zero elements.
  \end{proof}
\end{proposition}

\begin{proof}[Proof of Lemma~\ref{thm:zsignature}]
  The first inequality \eqref{eq:vM} follows from
  Theorem~\ref{thm:TPchanges} from the fact that
  $z= \ones - \Lambda V \alpha$ and that the matrix
  $(\ones \;\; \Lambda V)$ is totally positive. For that reason, then,
  \begin{align}\label{eq:uneq2}
    v_M(z) \leq v_m( \big(1 \;\; -\alpha^T\big) ) \leq n.
  \end{align}

  Next, assuming $\card(\psi)>n$, Proposition~\ref{thm:phiznull}
  implies that $\Psi z \neq 0$ and from Theorem~\ref{thm:TPchanges} it
  follows that
  \begin{align}\label{eq:uneq1}
    v_m(z) \geq v_m( \Psi z) \geq v_M(V^T\Psi z) = v_M(0) = n-1.
  \end{align} 
  Then, combining \eqref{eq:uneq2} and \eqref{eq:uneq1} gives
  \eqref{eq:vm0}.
  To show the last statement it is sufficient to note that, by
  Proposition~\ref{thm:tetraSign}, $z_m \geq 0$ and
  $(-1)^n z_1 \geq 0$ and that, when these elements are non-null, the
  minimal number of sign changes cannot be $n-1$.
\end{proof}

\subsection{The inverse mapping}
Until now, only the image $\mathcal{I}_\omega$ of the mapping
$\psi \to \omega$ has been considered. However, knowing the inverse
mapping $\omega \to \psi$ is more helpful if one is concerned with the
distribution of $\omega$ given that of $\psi$. In this subsection, the
level sets of that function, which is not bijective, are geometrically
characterised. Here, to keep the exposition short, not all the details
will be formally proven.

Consider the inverse of the function $z:\psi \to z(\psi)$, which
consists on the level sets
$z^{-1}(\hat{z}) = \{\psi \in \Real_+^m \;|\; z(\psi) = \hat{z} \}$,
with $\hat{z} \in \mathcal{I}_z$. 
As discussed in Remark~\ref{thm:zInverse}, the closure of the inverse
of $z$ maps each value of $z$ to the convex cone
\begin{align*}
  \mathcal{C}_z = \{ \psi \in \Real_+^m \;|\;  V^T Z \psi = 0 \}.
\end{align*}
Being the intersection of $m$ half-spaces and $n$ linear spaces, the
cone $\mathcal{C}_z$ is polyhedral and, then, it can be characterised
by a finite set of extremal rays. That is, any element
$\psi \in \mathcal{C}_z$ can be written as
$\psi = \sum_{i=1}^{k_z} t_i d^{(z)}_{i}$, with $t_i \geq 0$ and
$d^{(z)}_i \in \Real^m_+$ for $i=1,\ldots,k_z$. Alternatively, in a
more compact form, $\psi = D_z t$, $t \geq 0$, with
$D_z = ( d_1^{(z)} \; d_2^{(z)} \; \cdots \; d_{k_z}^{(z)} )$.

Now, since these rays belong to the boundary of $\mathcal{C}_z$, the
cardinality of each of them is smaller than $m$. Moreover, the minimal
cardinality that allows to satisfy the condition $V^T Z \psi =0$ is
$n+1$. Here, it is conjectured, but not proven, that there exists a
set of extremal directions $d^{(z)}_i$, $i=1,\ldots,k_z$, all with
cardinality $n+1$.  In that case, the sparsity pattern of these
vectors is completely determined by the signature of $z$.

Indeed, consider, firstly, the case $m=n+1$. As it has already been
shown, for $\mathcal{C}_z$ to not be degenerate it is necessary that
the last elements of $z$ is positive and that $z$ has exactly $n$ sign
changes.  In that case $\mathcal{C}_z$ consists in exactly one
dimensional ray:
$\mathcal{C}_z = \{ \gamma d^{(z)}_1 \;|\; \gamma \geq 0 \}$, with the
elements of $d^{(z)}$ strictly positive.
For the remaining cases, $n+1<m$, consider a vector
$\psi \in \mathcal{C}_z$ with cardinality $n+1$:
$\psi = S_\tau \tilde\psi$ for some $\tau \in [m,n+1]$ and
with $\tilde\psi$ having strictly positive elements. Then,
$\tilde\psi$ should satisfy the constraint
$V^TS_\tau \tilde{Z} \tilde\psi = 0$, where
$\tilde{Z} = S_\tau^T Z S_\tau$ is the diagonal matrix corresponding
to $\tilde{z} = S_\tau^T z$.  Now, in order to have $\tilde\psi>0$, it
is necessary that $\tilde{z}$ has $n$ sign changes and ends with a
non-negative element. This means that the sparsity pattern $\tau$ of
the direction $\psi$ needs to be consistent with that constraint.

Once the sparsity pattern $\tau$ is fixed, the directions can be
computed as a positive solution to linear system
$V^T Z S_\tau \psi = 0$.  Choosing a positive base for that solution
set provides a suitable set of extremal directions corresponding to
the sparsity pattern $\tau$.

\begin{example}
  As an example, assume that the signature of $z$ is $(-+-+++)$, which
  implies that $m=6$ and $n=3$. 
  Since all the vectors $Z d^{(z)}_i$, $i=1,\ldots,k_z$, need to have
  $n$ sign changes and cardinality $n+1$, the matrix $D_z$ has the
  sparsity structure, modulo columns permutations, given by
  \begin{align*}
    D_z &= 
    \small
    \pmx{
    \times & \times & \times \\
    \times & \times & \times \\
    \times & \times & \times \\
    \times & 0      & 0 \\
    0      & \times & 0 \\
    0      & 0      & \times 
    },
    &
    z &\overset{sign}{=} 
    \small
    \pmx{ - \\ + \\ - \\ + \\ + \\ + }.
  \end{align*}
\end{example}

\begin{example}\label{thm:exD}
  If $m=9$, $n=4$ and 
  $z \overset{sign}{=} (+--++-+++)$, then the sparsity structure of
  $D_z$ is given by
  \begin{align*}
    D_z &=
        {\small\setlength{\arraycolsep}{2.5pt}
    \left(
    \begin{array}{cccccccccccc}
    \times & \times & \times & \times & \times & \times & \times & \times & \times & \times & \times & \times \\
    \hline
    \times & \times & \times & \times & \times & \times & 0 & 0 & 0 & 0 & 0 & 0 \\
    0 & 0 & 0 & 0 & 0 & 0 & \times & \times & \times & \times & \times & \times  \\
    \hline
    \times & \times & \times & 0 & 0 & 0 & \times & \times & \times & 0 & 0 & 0 \\
    0 & 0 & 0 & \times & \times & \times & 0 & 0 & 0 & \times & \times & \times \\
    \hline
    \times & \times & \times & \times & \times & \times & \times & \times & \times & \times & \times & \times \\
    \hline
    \times & 0 & 0 & \times & 0 & 0 & \times & 0 & 0 & \times & 0 & 0 \\
    0 & \times & 0 & 0 & \times & 0 & 0 & \times & 0 & 0 & \times & 0 \\
    0 & 0 & \times & 0 & 0 & \times & 0 & 0 & \times & 0 & 0 & \times 
    \end{array}
    \right)},
    &
      z &\overset{sign}{=} 
          \small
      \pmx{
         + \\ \hline
         - \\ - \\ \hline
         + \\ + \\ \hline
         - \\ \hline
         + \\ + \\ + \\ 
      }.
  \end{align*}
\end{example}

Notice that, the number of extreme directions, $k_z$, is given by
the product of the lengths of sign-concordant sections of $z$. For
instance, in Example~\ref{thm:exD} $z$ has sections of lengths
$l = (1, 2, 2, 1,3)$, then $k_z = \prod_{j=1}^{n+1} l_j = 12$.
Since $z$ has $n$ sign changes, $l$, the vector with these lengths,
has $n+1$ elements which sums to $m$. Furthermore, $k_z \geq m-n$.

  Consider the $m \times k_z$ matrix $ZD_z$. The matrix $D_z$ can be
  chosen so that this matrix is only function of the signature of $z$
  and not of the values of its elements. Indeed, the columns of that
  matrix are extremal rays of the convex cone
  $\{ a \in \Real^m \;|\; V^T a = 0, a_i z_i \geq 0, i=1,\ldots,m \}$.
  Let call the resulting matrix $A(\sign(z))$.  Note also that
  $A(\sign(z))$ is constant inside each of the cells represented in
  Figures~\ref{fig:aGeom2} and~\ref{fig:aGeom3}.
  The cone $\mathcal{C}_z$ can then be written as
  $\mathcal{C}_z = \{ \psi = Z^{-1} A(\sign(z)) t \;|\; t\geq 0\}$.

\section{Discussion and examples}
\subsection{Negative Shrinkages}
Let assume that $y$ has cardinality $n$, that is $y = S_\tau y_\tau$
for some $\tau \in [m,n]$.  From \eqref{eq:shrinkSel} it follows that,
whenever $n$ is even and
$\lambda_i \gg \max\limits_{j \in \tau}\lambda_j$, the $i$-th
shrinkage factor can become negative (and eventually large). The most
extreme behaviour arises when $\tau$ selects the smallest eigenvalues
as considered in the following result.

\begin{lemma}\label{thm:extremeShrink}
  Let 
  $\tau = \{ m-n+1, \ldots, m-1, m\}$, 
  then the following bounds hold for the $i$-th element of
  $\omega_{(\tau)}$:
  \begin{align*}
    \omega_{(\tau),i} &< 1 - (c-1)^n < 1, 
    &
    \text{for } i &\not\in \tau  \text{ and } n \text{ even},
    \\
    \omega_{(\tau),i} &> 1 + (c-1)^n > 1, 
    &
    \text{for } i &\not\in \tau  \text{ and } n \text{ odd},
  \end{align*}
  with
  $c = \frac{\lambda_{m-n}}{\lambda_{m-n+1}} > 1$.
\end{lemma}

\begin{corollary}
  If $c > 2$ then $\omega_i < 0$ for $i \not\in \tau$.
\end{corollary}

In \cite{LingjaerdeChristophersen:2000} the authors resumed their
analysis reporting a table analogous to Table~\ref{tab:lingjaerde}.
However, they were not able to fill the bottom-left cell. Here, Lemma
\ref{thm:extremeShrink} allows to derive that result too. More
precisely, the scaling is smaller or larger than unity depending on
wether $n$ is even or odd.
\begin{table}[h]
  \caption{\footnotesize 
    The size of the $i$-th shrinkages $\omega_i$ as determined
    by $\lambda_i$ and $y_i$ for extreme cases.
  }
  \label{tab:lingjaerde}
  \begin{tabular}{lcc}
    \hline
    & $|y_i| \simeq 0$ & $|y_i|$ large \\
    \hline
    $\lambda_i$ small & $\omega_i \leq 1$ & $\omega_i \leq 1$ \\
    $\lambda_i$ large & $(-1)^n(\omega_i-1)  \leq 0$ & $\omega_i \simeq 1$ \\
       \hline
  \end{tabular}
\end{table}

Note that, since $\omega$ is a smooth function of $y$, a similar
situation arise whenever $y$ is near to that corner point. The
following numerical example, shows that large negative shrinkages can
be observed also in situations not so extreme as those prescribed by
the sufficient condition presented of Lemma~\ref{thm:extremeShrink}.

\begin{example}
  \label{thm:exNegative}

  Consider the following setup as an exemplification of the above
  statements. The regression model comprises a set of $m=5$
  explanatory variables with correlations given by
  $\rho_{ij} = e^{-\frac13|i-j|}$, $i,j=1,\ldots,5$.  Numerically,
  these correlations and the corresponding eigenvalues are given by
  \begin{align*}
    \rho_{1,:} &= \pmx{ 1 & 0.717 & 0.513 & 0.368 & 0.264 }
    \\
    \lambda &= \pmx{3.185 & 0.981 & 0.411 & 0.241 & 0.181 }.
  \end{align*}
  Note setting is not at all extreme. Indeed, the condition number of
  the correlation matrix is 17.6 and the mean absolute correlation is
  just 0.54.  
\begin{table}[tb]
  \caption{%
    Values for shrinkages and for the $\DoF$ estimators when 
    the $y$ has cardinality $n$.}
  \label{tab:DoF}

  \begin{tabular}[t]{ccrrrrrrr}
    \hline\vspace*{-2.2ex}\\
    $n$ & $\tau$ & 
    $\omega_{(\tau)}$ &&&&& $\widehat\gdof$ & $\widehat\gdof_{DP}$ \\[.5ex]
    \hline
    2 
& \{1,2\} &   1.00 &  1.00 & 0.49 & 0.30 & 0.23 &   3.03 &     3.67 \\
& \{1,3\} &   1.00 &  1.96 & 1.00 & 0.62 & 0.47 &   5.05 &     3.65 \\
& \{1,4\} &   1.00 &  3.12 & 1.61 & 1.00 & 0.76 &   7.50 &     0.05 \\
& \{1,5\} &   1.00 &  4.06 & 2.11 & 1.31 & 1.00 &   9.48 &    -5.70 \\
& \{2,3\} & -14.15 &  1.00 & 1.00 & 0.69 & 0.54 & -10.92 &  -224.84 \\
& \{2,4\} & -26.40 &  1.00 & 1.41 & 1.00 & 0.80 & -22.19 &  -745.85 \\
& \{2,5\} & -36.27 &  1.00 & 1.74 & 1.25 & 1.00 & -31.28 & -1384.77 \\
& \{3,4\} & -81.42 & -3.27 & 1.00 & 1.00 & 0.86 & -81.83 & -6807.00 \\
& \{3,5\} &-111.13 & -5.15 & 1.00 & 1.14 & 1.00 &-113.14 & -12605.62 \\ 
& \{4,5\} &-201.77 &-12.60 & 0.10 & 1.00 & 1.00 &-212.27 & -41297.69 \\[1.2ex]
      3 
& \{1,2,3\} &  1.00 &  1.00 &  1.00 & 0.71 & 0.57 &    4.28 &       4.73 \\
& \{1,2,4\} &  1.00 &  1.00 &  1.36 & 1.00 & 0.81 &    5.16 &       4.84 \\
& \{1,2,5\} &  1.00 &  1.00 &  1.64 & 1.23 & 1.00 &    5.88 &       4.53 \\
& \{1,3,4\} &  1.00 & -1.95 &  1.00 & 1.00 & 0.87 &    1.92 &      -3.73 \\
& \{1,3,5\} &  1.00 & -3.26 &  1.00 & 1.13 & 1.00 &    0.87 &     -13.13\\
& \{1,4,5\} &  1.00 & -8.41 &  0.22 & 1.00 & 1.00 &   -5.19 &     -84.13 \\
& \{2,3,4\} &185.97 &  1.00 &  1.00 & 1.00 & 0.89 &  189.85 &  -34208.14\\
& \{2,3,5\} &252.63 &  1.00 &  1.00 & 1.10 & 1.00 &  256.73 &  -63310.82\\
& \{2,4,5\} &456.04 &  1.00 &  0.48 & 1.00 & 1.00 &  459.52 & -207058.06 \\
& \{3,4,5\} &1369.96&  19.9 &  1.00 & 1.00 & 1.00 & 1392.86 & $-1.87\times 10^6$\\[1.2ex]
      4 
& \{1,2,3,4\} & 1.00 & 1.00 & 1.00 & 1.00 & 0.89 &     4.89 &    4.99 \\
& \{1,2,3,5\} & 1.00 & 1.00 & 1.00 & 1.10 & 1.00 &     5.10 &    4.99\\
& \{1,2,4,5\} & 1.00 & 1.00 & 0.55 & 1.00 & 1.00 &     4.55 &    4.79 \\
& \{1,3,4,5\} & 1.00 &14.07 & 1.00 & 1.00 & 1.00 &    18.07 & -165.86\\
& \{2,3,4,5\} &-3071.07&1.00& 1.00 & 1.00 & 1.00 & -3067.07 & $-9.44\times 10^6$ \\[.5ex]
      \hline
  \end{tabular}

\end{table}
Shrinkages for each corner point $\omega_{(\tau)}$ and for $n=2,3,4$
are shown in Table~\ref{tab:DoF}. That table reports also the value of
the statistics $\widehat\gdof$ and $\widehat\gdof_{DP}$ wich will be
introduced in the next section.

Here, for $n=3$ and $\tau=\{1,4,5\}$ a consistent negative expansion
($\omega_2 = -8.41$) occurs also in a situation different from the
ones suggested in the sufficient condition of Lemma
\ref{thm:extremeShrink}.

\end{example}

Recall that, by Theorem~\ref{thm:omegaAvg} the shrinkage vector
$\omega$ is a weighted average of the $\omega_{(\tau)}$ with weights
being function of the observation vector $y$. Therefore, the large
values of the shrinkages that arise at some of the points
$\omega_{(\tau)}$ may have a consistent effect even when the actual
$\omega$ is not to much near to these corner points.

\subsection{On the Degrees of Freedom of the PLS estimator}

Shrinkages represent natural tools for the study of the prediction
properties of the PLS estimator. Indeed, the sensitivities of the
prediction vector $\hat{y}$ with respect to changes on the actual
observations are given by the Jacobian
\begin{align}\label{eq:jac}
  J = \frac{\partial \hat{y}}{\partial y^T}
  &=
    (I - 2P)\Omega + 2P.
\end{align}
As a measure of that sensitivity, the Generalised DoF (GDoF) has been
introduced in \cite{Efron:2004,Ye:93} and is given by
$\gdof = \E[\tr(J)].$
The GDoF represents an extension of DoF concept to non linear
estimators. The use of that measure for PLS regressions was advocated
in \cite{KramerSugiyama:2011} where the authors propose its sample
counterpart as an unbiased estimator for $\gdof$. The need for an
estimator for the DoF arise, for instance, in the estimation of the
disturbance or prediction error variances or for determining the
number of directions to be used by PLS. The alternative measure
$\gdof_{DP} = \E[\tr(2J - J^TJ)]$, for the DoF of PLS have been
proposed in \cite{Denham:97,Phatak:2002}, while other measures based
on cross validation have also been considered (see for instance
\cite{VanDerVoet:1999}).

Consider now, the behaviour of this DoF estimator in the situations
identified in the previous subsection. Assume that $y$ has cardinality
$n$, that is $y = S_\tau y_\tau$ for an appropriate $\tau$. Under this
setup, the projection $P$ can be rewritten as $P = S_\tau S_\tau^T$
and the estimators for the degree of freedom become
\begin{align*}
  \widehat\gdof
  &= 
    \tr\Big(\Omega + 2 S_\tau S_\tau^T  (I-\Omega) \Big)
    = 
    n + \sum_{i \not\in \tau} \omega_i,
    \intertext{and}
  \widehat{GDoF}_{DP}
  &=
    \tr\Big( I - (I-\Omega)^2 \Big)
    = 
    m - \sum_{i \not\in \tau} (1-\omega_i)^2.
\end{align*}
Then, if $n$ is even and the assumptions of Lemma
\ref{thm:extremeShrink} are satisfied with $c>2$, then
$\widehat\gdof < n$.  Clearly, alternative values of $c$ could also
lead to a negative $\widehat\gdof$ or to a value exceeding the number
of observations.  Moreover, if the density of $y$ is concentrated
enough around $S_\tau \mu_\tau$, the same conclusions can be drawn for
the actual $\gdof$. Analogously, meaningless values
$\widehat\gdof_{DP}$ arise in proximity of ``degenerate'' models where
shrinkages can become very large in absolute value (see Table
\ref{tab:DoF} below).

Lemma~\ref{thm:extremeShrink} provides a sufficient condition that
contradicts a conjecture considered in \cite{KramerSugiyama:2011} and
``voiced'' in \cite{FrankFriedman:1993,MartensNaes:book} which states
that $\gdof > n$.  This property, was analytically proven to hold for
the first PLS step and empirically confirmed by means numerical
experiments for the remaining ones \cite{KramerSugiyama:2011}.
It is also worth noting that, negative values for $\widehat\gdof$ were
indeed observed in the experiment of Kramer and Sugiyama for large
model dimensions (see \cite{KramerSugiyama:2011} Section
4.3). Nonetheless, this behaviour was attributed to numerical
instabilities that characterises Krylov methods and which are likely
to occur for high-dimensional models.  Even if the numerical
instability of this class of methods is here recognised, we believe
that most the these negative DoFs are the effects a mechanism similar
to the one here discussed. Indeed, as $m$ increases it is more likely
to have $y$ orthogonal to a consistent set of principal axes, that is,
to observe a shrinkage vector located near a corner of the domain
region.

A clear example of this behaviour can be seen in Example
\ref{thm:exNegative}.  In that setup, computing the estimator for the
DoF when $y= (1,0,0,1,1)$ at the 3-rd iteration of the PLS gives an
estimate $\widehat\gdof = -5.18843$ (see Table~\ref{tab:DoF}). Looking
at Table~\ref{tab:DoF}, it is clear that things can go much worse,
this DoF measure can become extremely large and either positive or a
negative.

\begin{example}[continuation of Example~\ref{thm:exNegative}]
  The cases reported in Table~\ref{tab:DoF} are cases of exact
  under-specification. Indeed, when the cardinality of $y$ is equal to
  $n$, the PLS method has found the OLS estimator and thus there may
  be no reason to look at the sensitivity w.r.t. null elements of $y$.
  To analyse a less extreme situation, using the same model matrix,
  consider a setup where $y = \Lambda \beta$,
  with
  \begin{align*}
    \beta = \pmx{0.10 & 0.01 & 0.01 & 5.00 & 5.00 }.
  \end{align*}
  For that observation vector the estimates of the DoF for the 3rd
  step of PLS ($n=3$) is given by $\widehat\gdof = -3.134$. 

  Now, in order to consider a proper inferential setup, the above
  example is extended to the regression $y = \Lambda \beta + \eps$,
  where the additional disturbance vector is normally distributed:
  $\eps \sim N(0,\sigma^2 \Lambda)$, with $\sigma=0.02$.  Being unable
  to derive an explicit expression for the distribution of
  $\widehat\gdof$, a Monte Carlo (MC) experiment with 20000
  replications from the model has been performed. The resulting value
  for $\gdof = \E[\widehat\DoF]$ is -0.461 with an MC error having
  standard deviation of 0.026.  The result is less sharp than in the
  previous estimate, but still it is negative and consequently smaller
  than $n$.  The empirical distribution function of $\widehat\gdof$ is
  reported in Figure~\ref{fig:DoF} and MC estimation for the
  probability $\Prob[\widehat\gdof<0]$ is 0.56. Figure~\ref{fig:DoF}
  reports also the empirical distribution of $\widehat\gdof_{DP}$
  showing the pour performances of that estimator.

\begin{figure}[h]
  \centering

  \begin{tikzpicture}[scale=.7]
    \begin{axis}[
      ylabel={\footnotesize CDF},xlabel={\footnotesize $\widehat\gdof$},
      enlarge x limits=false,
      extra x ticks={0},
      extra y ticks={0},
      extra x tick style={grid=major}, 
      extra y tick style={grid=major}, 
      tick label style={font=\footnotesize,/pgf/number format/.cd,fixed,precision=3}
      ]
      \addplot[thick,blue!60!black]
      table {
  x     y
-5.2	0
-5.	0.0843
-4.8	0.1338
-4.6	0.1707
-4.4	0.1981
-4.2	0.2223
-4.	0.2475
-3.8	0.2689
-3.6	0.2882
-3.4	0.3076
-3.2	0.3252
-3.	0.3406
-2.8	0.3557
-2.6	0.3739
-2.4	0.3899
-2.2	0.4058
-2.	0.4208
-1.8	0.4355
-1.6	0.4495
-1.4	0.4637
-1.2	0.4783
-1.	0.4929
-0.8	0.5067
-0.6	0.5204
-0.4	0.5331
-0.2	0.5477
0.	0.5621
0.2	0.575
0.4	0.5894
0.6	0.6014
0.8	0.6148
1.	0.6291
1.2	0.6427
1.4	0.6565
1.6	0.6703
1.8	0.6859
2.	0.6988
2.2	0.71
2.4	0.722
2.6	0.7343
2.8	0.7477
3.	0.7592
3.2	0.7717
3.4	0.7854
3.6	0.7978
3.8	0.8112
4.	0.827
4.2	0.8395
4.4	0.8546
4.6	0.87
4.8	0.8848
5.	0.8993
5.2	0.9129
5.4	0.9255
5.6	0.9414
5.8	0.9562
6.	0.9695
6.2	0.9854
6.4	1
6.5	1
      };
    \end{axis}
  \end{tikzpicture}
  \qquad
  \begin{tikzpicture}[scale=.7]
    \begin{axis}[
      ylabel={\footnotesize CDF},xlabel={\footnotesize $\widehat\gdof_{DP}$},
      enlarge x limits=false,
      extra x ticks={0},
      extra y ticks={0},
      extra x tick style={grid=major}, 
      extra y tick style={grid=major}, 
      tick label style={font=\footnotesize,/pgf/number format/.cd,fixed,precision=3}
      ]
      \addplot[thick,blue!60!black]
      table {
  x     y
-85.	0
-84.	0.0105
-83.	0.06765
-82.	0.1106
-81.	0.14435
-80.	0.17325
-79.	0.197
-78.	0.22015
-77.	0.24025
-76.	0.2585
-75.	0.27535
-74.	0.2927
-73.	0.3085
-72.	0.32235
-71.	0.33735
-70.	0.3504
-69.	0.3643
-68.	0.37855
-67.	0.39245
-66.	0.4059
-65.	0.41875
-64.	0.43115
-63.	0.4434
-62.	0.4547
-61.	0.46615
-60.	0.4766
-59.	0.48855
-58.	0.49855
-57.	0.50875
-56.	0.5199
-55.	0.53035
-54.	0.5413
-53.	0.5512
-52.	0.56265
-51.	0.5718
-50.	0.5822
-49.	0.5928
-48.	0.6021
-47.	0.61145
-46.	0.62145
-45.	0.63115
-44.	0.64005
-43.	0.6489
-42.	0.6591
-41.	0.6692
-40.	0.6792
-39.	0.68925
-38.	0.6989
-37.	0.7089
-36.	0.71825
-35.	0.7259
-34.	0.73505
-33.	0.74365
-32.	0.75345
-31.	0.7624
-30.	0.7706
-29.	0.7795
-28.	0.7888
-27.	0.79785
-26.	0.80585
-25.	0.81505
-24.	0.82415
-23.	0.83325
-22.	0.84145
-21.	0.85095
-20.	0.85925
-19.	0.86885
-18.	0.87815
-17.	0.887
-16.	0.89635
-15.	0.9043
-14.	0.9122
-13.	0.91985
-12.	0.92625
-11.	0.9336
-10.	0.94195
-9.	0.94895
-8.	0.95605
-7.	0.96215
-6.	0.96705
-5.	0.97345
-4.	0.9786
-3.	0.9837
-2.	0.9887
-1.	0.99235
0.	0.9954
1.	0.99735
2.	0.9991
3.	0.99995
4.	1
      };
    \end{axis}
  \end{tikzpicture}
  \caption{\footnotesize CDF of the $\widehat\gdof$ and
    $\widehat\gdof_{DP}$ when $\lambda=(3.18,0.98,0.41,0.24,0.18)$,
    $\beta = (0.1,0.01,0.01,5,5)$, $\sigma=.02$ and $n=3$ and
    estimated by a MC experiment with $20\ 000$ replications. }
  \label{fig:DoF}
\end{figure}
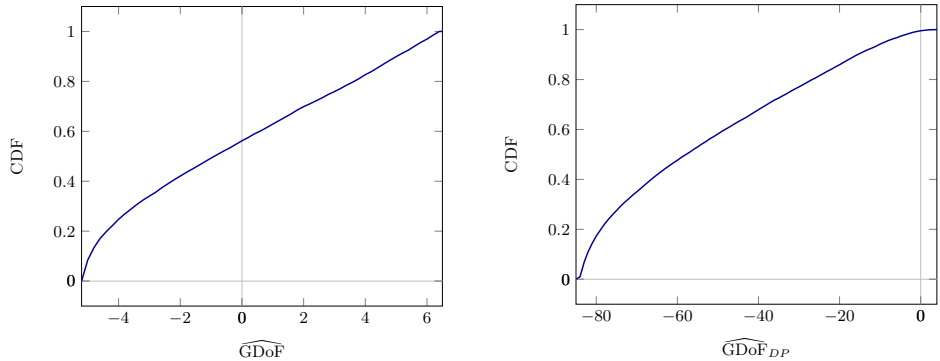  
\end{example}

\section{Conclusions}

A precise characterisation of the geometrical structure of PLS
shrinkages is here provided. The proposed analysis encompass and
complete a part of the literature on PLS regression
\cite{ButlerDenham:2000,FrankFriedman:1993,Goutis:96,Helland:1988,LingjaerdeChristophersen:2000,Kramer:2007}.
Here, the shrinkage vector is expressed as a weighted average of a set
of basic vectors that do not depend on the observations.  That
expression is a generalisation of the one considered in
\cite{ButlerDenham:2000} for a couple of special cases.  Also, this
analysis allowed to complete the one proposed in
\cite{LingjaerdeChristophersen:2000} where one extreme situation could
not be addressed. The explicit expression here proposed may represent
a starting point for the derivation of the distributions needed for
performing inference with PLS regression estimators. To this end, the
author remarks that the domain of the shrinkage vector variable has
been here formally and completely characterised. Moreover, the inverse
image of the shrinkage function is provided. This result provides a
starting step for the derivation of the distribution of the shrinkage
factors.

Furthermore, regions where the PLS regression estimator has an highly
non-linear behaviour and very large shrinkages (in absolute value)
have been characterised. In these situations, recently proposed
measures of the DoF for non-linear estimators completely fail as they
provide unrealistic results such as extremely large or negative
values.
The analytic tools here derived allowed to prove that the conjecture
stating that the PLS estimator always uses more ``DoF'' than the
number of PLS directions does not generally hold
\cite{FrankFriedman:1993,KramerSugiyama:2011,MartensNaes:book}. To
this end, a sufficient condition and some counterexamples have been
provided.
The failure of the ``GDoF'' statistic
\cite{Efron:2004,KramerSugiyama:2011,Ye:93} here identified points to
the need of a deeper reflection on the DoF concept especially in
highly non-linear contexts.

\bibliographystyle{imsart-number}

\end{document}